\documentclass{amsart}
\usepackage{amsmath,amssymb,amsthm,mathrsfs}

\numberwithin{equation}{section}

\newtheorem{thm}{Theorem}[section]
\newtheorem{prop}[thm]{Proposition}
\newtheorem{lem}[thm]{Lemma}

\theoremstyle{definition}

\newtheorem{remark}[thm]{Remark}

\newcommand{\sW}{{\mathcal W}}

\def\ga{\Gamma}

\def\l{\lambda}

\def\t{\tau}
\def\va{\varphi}

\def\tht{\Theta}

\def\z{\zeta}
\def\ts{\times}

\def\iy{\infty}
\def\im{{\rm Im\, }}

\def\kr{{\rm Ker\, }}

\def\rank{{\rm rank\, }}

\def\wt{\widetilde}

\def\BC{{\mathbb C}}
\def\BD{{\mathbb D}}

\def\BT{{\mathbb T}}

\newcommand{\mat}[2]{\ensuremath{\left[\begin{array}{#1}#2\end{array} \right]}}

\newcommand{\sbm}[1]{\left[\begin{smallmatrix} #1\end{smallmatrix}\right]}
\newcommand{\ands}{\quad\mbox{and}\quad}
\newcommand{\wtil}{\widetilde}
\newcommand{\half}{\frac{1}{2}}

\newcommand{\inn}[2]{\ensuremath{\left\langle #1,#2 \right\rangle}}

\newcommand{\nn}{\notag}

\begin{document}

\title[The Bezout-corona  problem revisited: Wiener space setting]{The Bezout-corona problem revisited:\\   Wiener space setting}

\author[G.J. Groenewald]{G.J. Groenewald}
\address{%
G.J. Groenewald, Department of Mathematics, Unit for BMI, North-West University\\
Private Bag X6001-209, Potchefstroom 2520, South Africa}
\email{Gilbert.Groenewald@nwu.ac.za}


\author[S. ter Horst]{S. ter Horst}

\address{%
S. ter Horst, Department of Mathematics, Unit for BMI, North-West University\\
Private Bag X6001-209, Potchefstroom 2520, South Africa}

\email{sanne.terhorst@nwu.ac.za}


\author[M.A. Kaashoek]{M.A. Kaashoek}

\address{%
M.A. Kaashoek, Department of Mathematics,
VU University Amsterdam\\
De Boelelaan 1081a, 1081 HV Amsterdam, The Netherlands}

\email{m.a.kaashoek@vu.nl}

\thanks{The third author gratefully thanks the mathematics department of North-West University, Potchefstroom campus, South Africa, for the generous support  during his visit May~22 -- June~12, 2014.}


\begin{abstract} The matrix-valued  {Bezout-corona} problem $G(z)X(z)=I_m$, $|z|<1$, is studied in a Wiener space setting, that is, the given function $G$ is an analytic matrix function on the unit  {disc} whose Taylor coefficients are absolutely summable and the same is required for the solutions $X$. It turns out that all Wiener solutions  can be described explicitly in terms of two matrices and a  square  {analytic} Wiener function $Y$ satisfying $\det Y(z)\not =0$ for all $|z|\leq 1$.  It is also shown that some  of the results   hold  in the $H^\iy$  {setting, but} not all.  In fact, if  $G$ is an $H^\iy$ function, then $Y$ is just an $H^2$  {function}. Nevertheless, in this case, using the two matrices and the function   $Y$,  all $H^2$ solutions to the  {Bezout-corona} problem can be described  explicitly in a form analogous to the one appearing in the Wiener setting.
\end{abstract}



\maketitle


\section{Introduction and main results}\label{intro}\setcounter{equation}{0}

Let  $G\in H_{m\ts p}^\iy$, that is, $G$ is an $m\ts p$ matrix function whose entries are $H^\iy$ functions on the open unit disc $\BD$.  The  \emph{$H^\iy$-corona problem} asks for a function $X\in H_{p\ts m}^\iy$  such that
\begin{equation}\label{corona}
G(z)X(z)=I_m\quad (z\in\BD).
\end{equation}
This problem has its roots in the paper \cite{Carl62} for the  case $m=1$,  and in \cite{Fuhr68} for the case $m>1$. Since then it has been studied in various contexts for which we refer to the books  {\cite{Helton87,Nikol86,Nikol02,Peller03}} and the recent papers  {\cite{FKR2a,FKR2b,Treil04,TW05,TrentZh06}}. See also the introduction of  \cite{FKR1} for the role of  equation \eqref{corona} in mathematical systems and control theory problems.  The problem is also closely related to the Leech problem  \cite{Leech} (see  also the comments in \cite{KaashRov14}) where the identity matrix $I_m$ in the right hand side of \eqref{corona} is replaced by another $H^\iy$ matrix function of appropriate size. In the Leech problem as well as in the corona problem norm constraints on the solution $X$ are often the main issue.  { When norm constraints are not the main issue one often refers to \eqref{corona} as a \emph{Bezout  problem} in a $H^\iy$ setting}.

We view the present  paper as an addition to the papers \cite{FKR2a} and \cite{FKR2b} which deal with the  {Bezout-corona} problem in the setting of stable rational matrix functions. Here we consider  equation \eqref{corona} in a Wiener space setting. We assume that  $G$ belongs  to  the Wiener space $\sW_+^{m\ts p}$ and we look for solutions $X$ which belong  to  the Wiener space $\sW_+^{p\ts m}$. In other words,    $G \in H_{m\ts p}^\iy$  and $X\in H_{p\ts m}^\iy$  and both have the additional property that their  Taylor coefficients at zero are absolutely summable. In this case we refer to \eqref{corona} as the \emph{ {Wiener-Bezout problem}}.  We shall be interested in the description of all Wiener solutions and the least square Wiener solution. The  {Wiener-Bezout} problem includes problem \eqref{corona}  for the case when $G$ is a stable rational matrix function and the solution $X$ is required to be stable rational matrix function too; see \cite{FKR2a} and \cite{FKR2b}.  For more information on Wiener spaces we refer the reader to  the final paragraph  of this introduction.

 Assuming $G\in H_{m\ts p}^\iy$, we shall also be interested in solutions $X$ to \eqref{corona}  that belong to $H_{p\ts m}^2$, where $H_{p\ts m}^2$ stands  for the linear spaces consisting of  all $p \ts m$ matrices   with entries in   $H^2$.  In that case we refer to \eqref{corona} as the  {$H^2$-Bezout} problem.

Recall, cf.,   \cite[Theorem 3.61]{Peller03} or \cite[Section 2]{FtHK-IEOT}, that the $H^\infty$-corona problem is solvable if and only if $T_G$ admits a right inverse. Here $T_G$ is the analytic Toeplitz operator
\[
T_G=\begin{bmatrix}
G_0&0&0&\cdots\\
 G_1&G_0&0&\cdots\\
  G_2&G_1&G_0&\cdots\\
  \vdots&\vdots&\vdots
\end{bmatrix}: \ell_+^2(\BC^p) \to  \ell_+^2(\BC^m),
\]
where $G_0, G_1, G_2, \ldots$ are the Taylor coefficients of $G$ at zero. Note that $T_G$ has  a right inverse if and only if  $T_GT_G^*$ is strictly positive. {Since $\sW_+^{p\ts m}\subset H^\iy_{p\ts m},$} for the   {Wiener-Bezout} problem to be solvable $T_GT_G^*$ has to be strictly positive.  We shall see that this condition is also sufficient and allows one to give a  description of all solutions to the  {Wiener-Bezout} problem in a simpler and more concrete form than for the general $H^\iy$-corona problem.

For our first main result  we need to introduce two matrices $\Xi_0 $ and $\tht_0 $, and a  $p\ts p$  matrix function $Y$ analytic on $\BD$ as follows.  Let $G\in H_{m\ts p}^\iy$, and assume that  $T_GT_G^*$ is strictly positive.   Then:
\begin{itemize}
\item [(M1)] $\Xi_0 $ is the $p\ts m$ matrix defined by $\Xi_0    = E_p^*T_G^*(T_GT_G^*)^{-1}E_m$;
\item [(M2)] $\tht_0 $ is the   $p\ts k$ matrix defined by
\begin{equation}\label{deftht0}
\Theta_0\Theta_0^* = I_p - E_p^*T_G^*(T_GT_G^*)^{-1}T_GE_p,\quad
\kr \Theta_0 = \{ 0 \}.
\end{equation}
\end{itemize}
Here for  any positive integer $n$ we write $E_n $ for the canonical embedding of $\BC^n$ onto the first coordinate space of $\ell_+^2(\BC^n)$, that is,
\begin{equation}
\label{defEn}
E_n  = \begin{bmatrix}
           I_n & 0 & 0 & 0 & \cdots \,\,\\
         \end{bmatrix}{}^\top:\mathbb{C}^n \rightarrow \ell_+^2(\mathbb{C}^n).
\end{equation}
Since $\kr \Theta_0 = \{ 0 \}$,   the integer $k$ in item (b) is equal to the rank of  the matrix $I_p - E_p^*T_G^*(T_GT_G^*)^{-1}T_GE_p$.  We shall see (Lemma \ref{L:MatRes} in the next section) that this rank is equal to $p-m$, even in  {the} $H^\iy$  {setting.} Finally, we define $Y$ to be the analytic $p\ts p$ matrix function on $\BD$ given by
\begin{align}\label{defY}
Y(z)=I_p-zE_p^*(I-zS_p^*)^{-1}T_G^*(T_GT_G^*)^{-1}H_GE_p\quad (z\in\BD).
\end{align}
Here for any positive integer $n$ the operator $S_n$ is the block forward shift on $\BC^n$. Furthermore,  $H_G$ is the Hankel operator defined by $G$, that is,
\[
H_G=\begin{bmatrix}
G_1&G_2&G_3&\cdots\\
 G_2&G_3&G_4&\cdots\\
  G_3&G_4&G_5&\cdots\\
  \vdots&\vdots&\vdots
\end{bmatrix}: \ell_+^2(\BC^p) \to  \ell_+^2(\BC^m).
\]
In other {words,} the Taylor coefficients of  $Y_0, Y_1, Y_2, \cdots $ of $Y$ at zero are given by
\begin{equation}
\label{defYj}
Y_0=I_p \ands \begin{bmatrix}Y_1\\ Y_2\\ \vdots\end{bmatrix}=-T_G^* (T_G T_G^*)^{-1} \begin{bmatrix}G_1\\ G_2\\ \vdots\end{bmatrix}.
\end{equation}
Note that the operator $T_G^*(T_GT_G^*)^{-1}$ appearing in the definitions of the matrices $\Xi$  and  {$\tht_0$} and the function $Y$ is the Moore-Penrose right inverse of $T_G$. In a less explicit form the function Y already appears in the  papers \cite{FKR2a, FKR2b}. The central role of this function  is a new aspect of the present paper.

Finally, with  the function $Y$ and the two matrices $\Xi$  and  {$\tht_0$} we associate
the  following two functions
\begin{equation}
\label{defXiTheta}
\Xi(z) =Y(z) \Xi_0 \ands \tht(z) =Y(z) \tht_0 \quad (z\in \BD).
\end{equation}

The next   theorem  is our  main result  in the Wiener space setting. It shows that with these three entities $\Xi_0$, $\tht_0$ and $Y$ all solutions to the  {Wiener-Bezout problem} can be described explicitly, and that the function $\Xi$ defined by  the first identity in  \eqref{defXiTheta} is the least squares solution.

\begin{thm}\label{thmmain1}
Let $G\in \sW_+^{m\ts p}$, and assume that  $T_GT_G^*$ is strictly positive. Then the matrix function $Y$ defined by \eqref{defY}  belongs to the Wiener space $\sW_+^{p\ts p}$, $\det Y(z)\not =0$ for each $|z|\leq 1$, and
\begin{align}\label{invY}
Y(z)^{-1}=I_p+zE_p^*T_G^*(T_GT_G^*)^{-1}H_G(I-z S_p)^{-1}E_p\quad (  {z\in\BD}).
\end{align}
In particular,  $Y^{-1}$ is a  Wiener function, and hence $Y$  is invertible outer.  Furthermore,
\begin{itemize}
\item[\textup{(i)}] $G(z)Y(z)=G_0$ for each $|z|\leq 1$,
\item[\textup{(ii)}]
the function $\tht$  defined by the second identity in \eqref{defXiTheta} belongs to $\sW_+^{p\ts (p-m)}$ \textup{(}in particular, $k=p-m$\textup{)}, and  $\tht$ is an  inner function  with $\im T_\tht=\kr T_G$,
\item[\textup{(iii)}]  the function  $H(z):=(\tht_0^*\tht_0)^{-1}\tht_0^*(I_p-\Xi_0G_0)  Y(z) ^{-1}$  belongs to the Wiener space $\sW_+^{(p-m)\ts p}$, and
\begin{equation}
\label{invGH1}
\det \begin{bmatrix} G(z)\\[.2cm] H(z) \end{bmatrix}\not =0 \ \mbox{and}\
\begin{bmatrix}
G(z)\\[.2cm]  H(z) \end{bmatrix}^{-1}=Y(z)\begin{bmatrix}\Xi_0 &\tht_0  \end{bmatrix} \quad (|z|\leq 1).
\end{equation}
\end{itemize}
Furthermore, for any $V\in \sW_+^{(p-m)\ts m}$ the function
\begin{equation}\label{allsolW}
X(z)=Y(z)\Xi_0+Y(z)\tht_0V(z)\quad (|z|\leq 1)
\end{equation}
is a solution to the  {Wiener-Bezout} problem associated with $G$, and all  {solutions are obtained} in this way. Moreover, with $X$ given by \eqref{allsolW} we have
\begin{equation}
\label{H2idW}
\|X(\cdot)u\|^2_{H^2_p}=\|Y(\cdot)\Xi_0 u\|^2_{H^2_p}+\|V(\cdot)u\|^2_{H^2_{p-m}}
\quad (u\in\BC^m).
\end{equation}
In particular, the function $\Xi(z)=Y(z)\Xi_0$ is the least squares solution to the  {Wiener-Bezout} problem associated with $G$.
\end{thm}

Item (iii) in the above theorem is closely related to Tolkonnikkov's lemma \cite{Tol81} (see also \cite[Appendix 3, item 10]{Nikol86}). In fact, from Tolkonnikkov's lemma it follows that \eqref{invGH1} holds true with $H$  {on the unit circle $\BT$} being given by
\begin{equation}
\label{TolH}
H(\z)=\tht^*\big(\z)(I_p-\Xi(z)G(\z)\big) \quad (\z\in \BT).
\end{equation}
At the end of   Section \ref{secWiener} (see Remark \ref{remHwtH}) we shall show that  the function $H$ defined  by the above  formula  and the   function $H$   defined in item (iii) of the above theorem are one and the same function. Specifying \eqref{invGH1} for $z=0$ we see that
\begin{equation}
\label{invXiTheta1}
\mat{cc}{\Xi_0 & \Theta_0}^{-1}=\mat{c}{G_0\\ H_0}\quad\mbox{with}\quad H_0=(\Theta_0^*\Theta_0)^{-1}\Theta_0^*(I_p-\Xi_0 G_0).
\end{equation}
Lemma \ref{L:MatRes} in the next section shows that this inversion formula remains true if $G$ is just an $H^\iy$ function.

Theorem  \ref{T:H2cor}, which is our second main result, presents a   (partial) analogue of  Theorem \ref{thmmain1} in an $H^\iy/H^2 $ setting. Let $G\in H_{m\ts p}^\iy$, and assume  that $T_GT_G^*$ to be strictly positive. Then the function $Y$ is still well defined on the open unit  {disc} $\BD$ and $\det Y(z)\not =0$ for each $z\in \BD$. However,  in general,  the  entries of $Y$ and $Y^{-1}$ are just  $H^2$ functions,  and formula \eqref{allsolW} yields $H^2$ solutions rather than $H^\iy$ solutions.  Moreover, if the free parameter $V$ in \eqref{allsolW} is taken from $H_{(p-m)\ts m}^2$,  then all $H^2$ solutions are obtained by \eqref{allsolW} and the $H^2$ norm in the identity \eqref{H2idW} appears in a natural way.

Finally, in Section \ref{Hinfty} we shall prove that  item (ii) carries over to an $H^\iy$ setting (see Proposition \ref{P:Theta}). The fact that $\tht$ is inner  with $\im T_\tht=\kr T_G$ follows from Lemma 2.1 in \cite{FtHK-8}. A more direct proof is given at the  end of Section~\ref{Hinfty}.  The statement that $k=p-m$ is new in the $H^\iy$ setting. For the proof see the final part of Lemma \ref{L:MatRes}.

The paper consists of five  sections, including the present introduction. In  the second section we present a  number of auxiliary results which are all valid in the $H^\iy$ setting.  Section \ref{secWiener} contains the proof  of Theorem  \ref{thmmain1}. Section~\ref{HiyH2} deals with the role of the function $Y$ in the $H^\iy$ case and presents  a partial analogue of Theorem \ref{thmmain1}, including the description of all $H^2$ solutions. In the final section we present a few concluding remarks  and   compute the function $Y$ for the case when $G(z)=\begin{bmatrix} 1+z & -z \end{bmatrix}$.

\paragraph{Notation and terminology.}
By  $\sW$ we denote the Wiener space (cf., item (a)  in \cite[Section XXIX.2]{GGK2}) consisting of all functions on the unit circle that have an absolutely summable Fourier expansion, and   $\sW^{r\ts s}$ stands for the linear space of all ${r\ts s}$ matrix functions of which the entries belong to $\sW$. Thus
\[
F\in \sW^{r\ts s} \ \Longleftrightarrow\  F(e^{it})=\sum_{\nu=-\iy}^\iy F_\nu  e^{ it\nu},  \ \mbox{where}\  \sum _{\nu=-\iy}^\iy \|F_\nu\|<\iy.
\]
As usual we refer to $F_\nu$ as   the $\nu$-th  Fourier coefficient of $F$. We also need the space  $\sW_+^{r\ts s}$ which consists of all  $F\in \sW^{r\ts s}$ that have an analytic extension to the open unit disc $\BD$, that is,
\[
F\in \sW_+^{r\ts s} \ \Longleftrightarrow\  F(e^{it})=\sum_{\nu=0}^\iy F_\nu  e^{ it\nu},  \ \mbox{where}\  \sum _{\nu=0}^\iy \|F_\nu\|<\iy.
\]
Each $F\in \sW^{r\ts s}$ is continuous on the unit circle,  and therefore each  $F\in \sW^{r\ts s}$ defines a (block) Toeplitz  operator $T_F$   mapping $\ell^2_+(\BC^s)$ into $\ell^2_+(\BC^r)$.  With $F\in \sW^{r\ts s}$, we associate the function $F^*\in \sW^{s\ts r}$ defined by $F^*(z)=F(1/\bar{z})^*$ for each $z\in\BT$. Then $T_{F^*}=T_F^*$.  Finally, note that $\sW_+^{r\ts s}\subset H_{r\ts s}^\iy\subset H_{r\ts s}^2$, where  $H_{r\ts s}^\iy$ and $H_{r\ts s}^2$ stand  for the linear spaces consisting of  all $r \ts s$ matrices   with entries in $H^\iy$ and $H^2$, respectively.

\section{Auxiliary results in an $H^\infty$ setting}\label{Hinfty}
\setcounter{equation}{0}

Throughout this section let $G\in H^\infty_{m\ts p}$ and assume that  $T_GT_G^*$ is strictly positive. We shall be dealing  with  the function $Y$ defined by \eqref{defY} and  the matrices  $\Xi_0$ and  $\tht_0$ defined by items (M1) and (M2) in the previous section.  Note that  the function $Y$  and  the matrices $\Xi_0$ and $\tht_0$ are well defined when $G\in H^\infty_{m\ts p}$ and $T_GT_G^*$ is strictly positive; it is not required for this that $G$ belongs to a Wiener space.

In this section we shall derive a number of auxiliary results that will be useful in proving Theorem \ref{thmmain1} in Section \ref{secWiener}. These auxiliary results will also allow us to present a partial generalization of Theorem \ref{thmmain1} in a $H^\iy/H^2$  {context in Section \ref{HiyH2}.}

The first result only involves the matrices $G_0$, $\Xi_0$ and $\tht_0$.

\begin{lem}\label{L:MatRes}
Let $\Xi_0$ and $\Theta_0$ be as in items \textup{(M1)} and \textup{(M2)} in the previous section. Then the  matrix $\begin{bmatrix} \Xi_0&\Theta_0\end{bmatrix}$ is invertible with inverse given by
\begin{equation}
\label{invXiTheta}
\mat{cc}{\Xi_0 & \Theta_0}^{-1}=\mat{c}{G_0\\ H_0}\quad\mbox{with}\quad H_0=(\Theta_0^*\Theta_0)^{-1}\Theta_0^*(I_p-\Xi_0 G_0).
\end{equation}
In particular, we have $k=p-m$ and $\im \tht_0=\kr G_0$.
\end{lem}

\begin{proof}[\bf Proof]
Note that $G_0=E_m^*T_G E_p$ and $G_0 E_p^*=E_m^*T_G$. Hence $E_m^*T_G=E_m^*T_GE_pE_p^*$.  These identities give
\begin{align*}
G_0\Xi_0&=G_0E_p^*T_G^*(T_GT_G^*)^{-1}E_m\\
&=E_m^*T_G  T_G^*(T_GT_G^*)^{-1}E_m=E_m^*E_m=I_m,
\end{align*}
and
\begin{align*}
G_0\Theta_0\Theta_0^*&=E_m^*T_G E_p(I_p-E_p^* T_G^*(T_GT_G^*)^{-1}T_GE_p)\\
&=(E_m^*-E_m^*T_G E_pE_p^* T_G^*(T_GT_G^*)^{-1})T_GE_p\\
&=(E_m^*-E_m^*T_G T_G^*(T_GT_G^*)^{-1})T_GE_p=(E_m^*-E_m^*)T_GE_p=0.
\end{align*}
Since $\im \Theta_0^*=\BC^k$, the latter implies $G_0\Theta_0=0$.

With the identities $G_0\Xi_0=I_m$ and $G_0\Theta_0=0$ we obtain
\[
(I-\Xi_0 G_0)\Xi_0=\Xi_0-\Xi_0=0\ands
(I-\Xi_0 G_0)\Theta_0=\Theta_0-0=\Theta_0.
\]
Note that $(\Theta_0^*\Theta)^{-1}\Theta_0^*$ is a left inverse of $\Theta_0$. Hence
\[
H_0\Xi_0=0\ands H_0\Theta_0=I_k.
\]
Combining the above identities shows
\begin{equation}
\label{ids2}
\mat{cc}{\Xi_0 & \Theta_0} \mat{c}{G_0\\ H_0}=I_p\ands
\mat{c}{G_0\\ H_0}\mat{cc}{\Xi_0 & \Theta_0}=\mat{cc}{I_m& 0\\0& I_k}.
\end{equation}
It follows that $\begin{bmatrix}\Xi_0 & \tht_0\end{bmatrix}$ is invertible and  {that} its inverse is given by \eqref{invXiTheta}.   In particular, $\begin{bmatrix}\Xi_0 & \Theta_0\end{bmatrix}$ is a square matrix, which implies $p=m+k$. Hence $k=p-m$. Moreover, $G_0\tht_0=0$ implies $\im \tht_0\subset \kr G_0$. We have $\kr\tht_0=\{0\}$, so that $\rank \tht_0=p-m$. Hence $\dim \im\tht_0=p-m$. On the other hand, we have $\im G_0=\BC^m$, which implies $ \dim \kr G_0=p-m$. Therefore $\im \tht_0= \kr G_0$.
\end{proof}

Lemma \ref{L:MatRes} can be seen as the special case of Proposition \ref{P:YH} below where $z=0$. To derive the later result we require the following observation about the function $Y$.

\begin{prop}\label{propY2}
Let $G\in H^\infty_{m\ts p}$ and assume that $T_GT_G^*$ is strictly positive. Then the function $Y$  defined by \eqref{defY}  is analytic on $\BD$, $\det Y(z)\not =0$ for each $z\in \BD$, and
\begin{equation}\label{invY2}
Y(z)^{-1}=I_p+zE_p^*T_G^*(T_GT_G^*)^{-1}H_G(I-z S_p)^{-1}E_p\quad (z\in\BD).
\end{equation}
In particular, the function $Y(\cdot )^{-1}$ is analytic on $\BD$. Moreover, we have
\begin{equation}\label{GYG0}
G(z)Y(z)=G_0 \quad (z\in\BD).
\end{equation}
\end{prop}

\begin{proof}[\bf Proof.]
That fact that $S_p$ has spectral radius equal to 1, yields that $Y$ is analytic on $\BD$. Since $S_p^*T_G^*=T_G^*S_m^*$ we can rewrite $Y$ as
\begin{align*}
Y(z)
&=I_p-zE_p^*(I-zS_p^*)^{-1}T_G^*(T_GT_G^*)^{-1}H_GE_p\\
&=I_p-zE_p^*T_G^*(I-zS_m^*)^{-1}(T_GT_G^*)^{-1}H_GE_p\\
&=D+zC(I-zA)^{-1}B,
\end{align*}
where in the last identity
\[
A=S_m^*,\quad B=(T_GT_G^*)^{-1}H_GE_p,\quad C=-E_p^*T_G^*,\quad D=I_p.
\]
Note that $H_GE_p=S_m^*T_G E_p$ and
\begin{align*}
S_m^*T_GE_pE_p^*T_G^*
&=S_m^*T_G(I-S_pS_p^*)T_G^*
=S_m^*T_GT_G^*-S_m^*S_mT_GT_G^*S_m^*\\
&=S_m^*T_GT_G^*-T_GT_G^*S_m^*.
\end{align*}
This yields that $A^\ts:=A-BD^{-1}C$ can be written as
\begin{align*}
A^\ts
&=S_m^*+(T_GT_G^*)^{-1}H_GE_pE_p^*T_G^*
=S_m^*+(T_GT_G^*)^{-1}S_m^*T_GE_pE_p^*T_G^*\\
&=S_m^*+(T_GT_G^*)^{-1}(S_m^*T_GT_G^*-T_GT_G^*S_m^*)
=(T_GT_G^*)^{-1}S_m^*T_GT_G^*.
\end{align*}
Thus $A^\ts$ is similar to $S_m^*$, and hence has spectral radius equal to 1. Then, by standard state space inversion results, cf., Theorem 2.1 in \cite{BGKR08} (with $\l=1/z$), it follows that $Y(z)$ is invertible for each $z\in\BD$ with inverse given by
\begin{align*}
Y(z)^{-1}
&=D^{-1}-zD^{-1}C(I-zA^\ts)^{-1}BD^{-1}\\
&=I+zE_p^*T_G^*(I-z(T_GT_G^*)^{-1}S_m^*T_GT_G^*)^{-1}(T_GT_G^*)^{-1}H_GE_p\\
&=I+zE_p^*T_G^*(T_GT_G^*)^{-1}(I-zS_m^*)^{-1}H_GE_p.
\end{align*}
Since  $S_m^*H_G=H_GS_p$,  we have $(I-zS_m^*)^{-1}H_G=H_G(I-zS_p)^{-1}$, and hence \eqref{invY2} holds. Note that the spectral radius of $S_p$ is equal to 1, which implies that the function $Y(\cdot )^{-1}$ is analytic on $\BD$.

Finally, we prove that \eqref{GYG0} holds. Let $Y_0, Y_1, Y_2, \ldots$ be the Taylor coefficients of $Y$ at zero.  {As observed in \eqref{defYj}, we have} $Y_0=I_p$ and
\begin{equation}
\label{defY2a}
\begin{bmatrix}Y_1\\ Y_2\\ \vdots\end{bmatrix}=-T_G^* (T_G T_G^*)^{-1} \begin{bmatrix}G_1\\ G_2\\ \vdots\end{bmatrix}, \quad \mbox{and hence}\quad  T_G\begin{bmatrix}Y_1\\ Y_2\\ \vdots\end{bmatrix}=-\begin{bmatrix}G_1\\ G_2\\ \vdots\end{bmatrix}.
\end{equation}
The latter identity  is equivalent to
\[
G(z)\left(\frac{Y(z)-I_p}{z}\right)=-\frac{G(z)-G_0}{z}  \quad (z\in \BD).
\]
Multiplying both sides of the above identity by $z$ and adding $G(z)$ on either side yields \eqref{GYG0}.
\end{proof}

\begin{prop}\label{P:YH}
Let $G\in H_{m\ts p}^\iy$ and assume that $T_GT_G^*$ is strictly positive. Let $Y$ be the function defined by \eqref{defY}, and define the functions $\Xi$ and $\tht$ by \eqref{defXiTheta}, with $\Xi_0$ and $\Theta_0$ the matrices in items \textup{(M1)} and \textup{(M2)} of the previous section. Consider the matrix function $H$ defined by
\begin{equation}\label{defH}
H(z)= H_0Y(z)^{-1},\ z\in\BD,\quad \mbox{with}\quad H_0=(\tht_0^*\tht_0)^{-1}(I_p-\Xi_0G_0).
\end{equation}
Then $H$ is analytic on $\BD$,
\begin{equation}
\label{invGH2}
\det \begin{bmatrix} G(z)\\[.2cm] H(z) \end{bmatrix}\not =0 \quad \mbox{and}\quad
\begin{bmatrix}
G(z)\\[.2cm]  H(z) \end{bmatrix}^{-1}=\begin{bmatrix}\Xi(z) &\tht(z)  \end{bmatrix} \quad (z\in \BD).
\end{equation}
\end{prop}

\begin{proof}[\bf Proof.]
Since $Y$ is analytic on $\BD$, clearly $H$ defined by \eqref{defH} is analytic on $\BD$. Furthermore, using Proposition \ref{propY2} we find that $G(z)=G_0 Y(z)^{-1}$, $z\in\BD$. Thus
\[
\mat{c}{G(z)\\ H(z)}=\mat{c}{G_0\\ H_0}Y(z)^{-1},\quad \mat{cc}{\Xi(z)& \tht(z)}=Y(z)\mat{cc}{\Xi_0 & \tht_0}\quad (z\in\BD).
\]
This shows that our claim reduces to the case $z=0$, which was proved in Lemma \ref{L:MatRes}.
\end{proof}

We conclude with two auxiliary results, the first is about the function $\Xi$ and the second about  $\tht$.

\begin{lem}\label{lemXi2}
Let $G\in H_{m\ts p}^\iy$ and assume that $T_GT_G^*$ is strictly positive. Then the function $\Xi$  defined by the first part  of \eqref{defXiTheta} is also given by
\begin{equation}
\label{altdefXi}
\Xi(z) =E_p^*(I-zS_p^*)^{-1}T_G^*(T_GT_G^*)^{-1}E_m\quad (z\in \BD)
\end{equation}
\end{lem}

\begin{proof}[\bf Proof]
Recall that $\Xi_0=E_p^*T_G^*(T_GT_G^*)^{-1}E_m$ This yields
\begin{align*}
H_G E_p\Xi_0
&=S_m^*T_GE_pE_p^*T_G^*(T_GT_G^*)^{-1}E_m\\
&=S_m^*T_G(I-S_pS_p^*)T_G^*(T_GT_G^*)^{-1}E_m\\
&=S_m^*E_m-S_m^*S_mT_GT_G^*S_m^*(T_GT_G^*)^{-1}E_m\\
&=-T_GT_G^*S_m^*(T_GT_G^*)^{-1}E_m.
\end{align*}
With this observation we obtain that the function $\Xi$ is also given by
\begin{align*}
\Xi(z)
&=Y(z)\Xi_0
=\Big(I-zE_p^*(I-zS_p^*)^{-1}T_G^*(T_GT_G^*)^{-1}H_GE_p\Big)\Xi_0\\
&=\Xi_0+zE_p^*(I-zS_p^*)^{-1}T_G^*S_m^*(T_GT_G^*)^{-1}E_m\\
&=E_p^*\Big(I+z(I-zS_p^*)^{-1}S_p^*\Big)T_G^*(T_GT_G^*)^{-1}E_m\\
&=E_p^*(I-zS_p^*)^{-1}T_G^*(T_GT_G^*)^{-1}E_m.
\end{align*}
This proves \eqref{altdefXi}.
\end{proof}

\begin{prop}\label{P:Theta}
Let $G\in H_{m\ts p}^\iy$ and assume that $T_GT_G^*$ is strictly positive. The function $\tht$ defined in \eqref{defXiTheta} belongs to $H^\iy_{p\ts (m-p)}$ and is an inner function with $\im T_\tht=\kr T_G$.
\end{prop}

\begin{proof}[\bf Proof.]
Using the definition of $Y$ in \eqref{defY} and the fact $H_GE_p=S_m^*T_GE_p$, we see that $\tht$ is also given by
\[
\tht(z)=(I_p-zE_p^*(I-zS_p^*)^{-1}T_G^*(T_G T_G^*)^{-1}S_m^*T_GE_p)\tht_0\quad (z\in\BD).
\]
By comparing this formula with \cite[Eq. (2.1)]{FtHK-8} we conclude that $\tht$ coincides (up to multiplication with a constant unitary matrix from the right) with  the inner function $\wt{\tht}$ satisfying $\im T_G^*= \ell_+^2(\BC^p)\ominus T_{\wt{\tht}} \ell_+^2(\BC^k)$, where $k$ is the number of columns of the matrix $\wt{\tht}(0)$. The existence of $\wt{\tht}$ is guaranteed by the Beurling-Lax theorem. Since $\kr T_G= (\im T_G^*)^\perp$, we conclude that $\kr T_G=\im T_\tht$. Finally, that $k=p-m$, and thus $\tht\in H^\iy_{p\ts (m-p)}$, follows from Lemma \ref{L:MatRes}.
\end{proof}

Note the proof of Proposition \ref{P:Theta} relies heavily on \cite[Lemma 2.1]{FtHK-8}. We also add something to the observations made in Section 2 of \cite{FtHK-8}, namely that $k=p-m$, i.e., $\tht\in H^\iy_{p\in(p-m)}$. This was proved in \cite[Lemma 2.2]{FKR2a} for the case that $G$ is a rational matrix function. We show here that the observation extends to the  non-rational case. Next we  {give a more direct proof of} Proposition \ref{P:Theta}.

\paragraph{Direct proof of Proposition \ref{P:Theta}.} Let $\tht$ be the analytic matrix function on $\BD$ defined by  the second identity in \eqref{defXiTheta}. We already know (see the final part of Lemma  \ref{L:MatRes} that $\tht_0$ has size $p\ts (p-m)$,  and hence $\tht$ is a matrix function of size $p\ts (p-m)$. To prove that $\tht$ is inner, let $ \ga_j$  be  $j$-th column of the block Toeplitz matrix  defined by $\tht$. Thus
\begin{equation}
\label{defgas}
\begin{bmatrix} \ga_0& \ga_1& \ga_2&\ \cdots \end{bmatrix}=\begin{bmatrix}
Y_0\tht_0& 0& 0& \cdots \\
Y_1\tht_0&Y_0\tht_0& 0&  \cdots \\
Y_2\tht_0&Y_1\tht_0&Y_0\tht_0 &\\
\vdots&\vdots&&\ddots
\end{bmatrix}.
\end{equation}
Note that $ \ga_0, \ga_1,  \ga_2,  \cdots$ are bounded linear operators from $\BC^{p-m}$ into $\ell_+^2(\BC^p)$.  This follows from the first identity in \eqref{defY2a}, the fact that $T_G^*(T_GT_G^*)^{-1}$ is a bounded operator from $\ell_+^2(\BC^m)$ into $\ell_+^2(\BC^p)$, and the fact that the first collumn of $H_G$ is a bounded operator from $\BC^p$  {into $\ell_+^2(\BC^m)$.  To prove that $\tht$ is inner it suffices to show that
\begin{itemize}
\item[(C1)]   $\ga_j$ is an isometry mapping   $\BC^{p-m}$   into $\ell_+^2(\BC^p)$ for $j=0, 1, 2, \dots$;
\item[(C2)] $\im  \ga_j \perp \im  \ga_k$  for $k\not =j$.
\end{itemize}
To see this, assume that both conditions are satisfied. Then the operator $T$ defined by be the  infinite  block  lower triangular   matrix on  the right hand side of \eqref{defgas} is an isometry mapping $\ell_+^2(\BC^{p-m})$ into $\ell_+^2(\BC^{p-m})$.   Moreover, $ S_pT=TS_{p-m}$.  It follows that  $T$ is a  Toeplitz operator,  and its defining function  $\tht(\cdot)=Y(\cdot)\tht_0$ belongs to $H_{p\ts (p-m)}^\iy$; cf., \cite[Section XXIII.3]{GGK2}. Thus $T=T_\tht$,  and $\tht$ is inner because  $T_\tht=T$ is an isometry, \cite[Section XXVI.3]{GGK2} or   \cite[Proposition 2.6.2]{FB10}.

In order to show that  conditions (C1) and (C2) are satisfied we need the following two  lemmas.

\begin{lem}\label{lemtht1}
Let $G\in H_{m\ts p}^\iy$, and assume that  $T_GT_G^*$ is strictly positive. Then
\begin{equation*}
\sum_{i=0}^\iy Y_i^*Y_{i+j}=
\left\{
\begin{array}{cl}
I_p +E_p^*H_G^*(T_GT_G^*)^{-1}H_G E_p &\mbox{when $j=0$},\\[.3cm]
-G_0^*E_m^*(S_m^*)^{j-1}(T_GT_G^*)^{-1}H_G E_p &\mbox{when $j= 1,2, \ldots $}.
\end{array}
\right.
\end{equation*}
\end{lem}

\begin{proof}[\bf Proof]
Note that
\begin{align}
&\sum_{i=0}^\iy Y_i^*Y_{i+j}=Y_0^* Y_j+ \begin{bmatrix} Y_1^* &  Y_2^*&  \cdots \end{bmatrix} \begin{bmatrix} Y_{j+1}\\    Y_{j+2}\\  \vdots \end{bmatrix}\nn\\
&\hspace{.5cm}=Y_0^* Y_j+  E_p^*H_G^*(T_GT_G^*)^{-1}T_G  (S_p^*)^j T_G^*(T_GT_G^*)^{-1}H_G E_p\nn\\
&\hspace{.5cm}=Y_0^* Y_j+  E_p^*H_G^*(T_GT_G^*)^{-1}T_G  T_G^*(S_m^*)^j  (T_GT_G^*)^{-1}H_G E_p\nn\\
&\hspace{.5cm}=Y_0^* Y_j+  E_p^*H_G^*(S_m^*)^j (T_GT_G^*)^{-1}H_G E_p, \quad j=0, 1, \ldots. \label{Y*Y}
\end{align}
Using $Y_0=I_p$ we see that with    $j=0$ the identity \eqref{Y*Y} yields the first part of the lemma.

Next assume that $j>0$. Recall that  $H_G E_p=S_m^*E_p$. Taking adjoints in the latter identity and  {using} $Y_0=I_p$ again, we see that \eqref{Y*Y} can be rewritten as
\begin{align*}
\sum_{i=0}^\iy Y_i^*Y_{i+j}&=Y_j +E_p^*T_G^*S_m(S_m^*)^j (T_GT_G^*)^{-1}H_G E_p\\
&=Y_j +E_p^*T_G^*S_mS_m^*(S_m^*)^{j -1}(T_GT_G^*)^{-1}H_G E_p\\
&=C_1-C_2,
\end{align*}
where
\begin{align*}
C_1&= E_p^*T_G^*(S_m^*)^{j -1}(T_GT_G^*)^{-1}H_G E_p\\
&= E_p^*(S_p^*)^{j -1}{T_G^*}(T_GT_G^*)^{-1}H_G E_p=-Y_j,
\end{align*}
and
\begin{align*}
C_2&=E_p^*T_G^*E_m E_m^*(S_m^*)^{j -1}(T_GT_G^*)^{-1}H_G E_p\\
&=G_0 E_m^*(S_m^*)^{j -1}(T_GT_G^*)^{-1}H_G E_p.
\end{align*}
Thus
\begin{align*}
\sum_{i=0}^\iy Y_i^*Y_{i+j}&=Y_j +C_1-C_2=-G_0 E_m^*(S_m^*)^{j -1}(T_GT_G^*)^{-1}H_G E_p.
\end{align*}
This proves the second part of the lemma.
\end{proof}

\begin{lem}\label{lemtht2}
Let $G\in H_{m\ts p}^\iy$, and assume that  $T_GT_G^*$ is strictly positive. Then
\begin{equation}
\label{eqtht0*}
\tht_0^*\Big(I_p +E_p^*H_G^*(T_GT_G^*)^{-1}H_G E_p\Big)\tht_0=I_{p-m}.
\end{equation}
\end{lem}

\begin{proof}[\bf Proof]
Using the definition of $\tht_0\tht_0^*$ in \eqref{deftht0} we see that
\begin{equation*}
E_p^*H_G^*(T_GT_G^*)^{-1}H_G E_p\tht_0\tht_0^*=A-B,
\end{equation*}
where
\begin{align*}
A&= E_p^*H_G^*(T_GT_G^*)^{-1}H_G E_p,\\
B&=E_p^*H_G^*(T_GT_G^*)^{-1}H_G E_pE_p^*T_G^*(T_GT_G^*)^{-1}T_GE_p\\
&=E_p^*H_G^*(T_GT_G^*)^{-1}S_m^*T_G E_pE_p^*T_G^*(T_GT_G^*)^{-1}T_GE_p.
\end{align*}
Here we used that $H_GE_p=S_m^*T_G E_p$. Next using $E_pE_p^*=I-S_pS_p^*$ we  write $B$ as $B=B_1-B_2$, where
\begin{align*}
B_1&=E_p^*H_G^*(T_GT_G^*)^{-1}S_m^*T_GT_G^*(T_GT_G^*)^{-1}T_GE_p\\
&=E_p^*H_G^*(T_GT_G^*)^{-1}S_m^*T_GE_p\\
&=E_p^*H_G^*(T_GT_G^*)^{-1}H_GE_p=A,
\end{align*}
and
\begin{align*}
B_2&= E_p^*H_G^*(T_GT_G^*)^{-1}S_m^*T_G S_pS_p^*T_G^*(T_GT_G^*)^{-1}T_GE_p\\
&= E_p^*H_G^*(T_GT_G^*)^{-1}S_m^* S_m T_GT_G^*S_m^*(T_GT_G^*)^{-1}T_GE_p\\
&= E_p^*H_G^*(T_GT_G^*)^{-1}T_GT_G^*S_m^*(T_GT_G^*)^{-1}T_GE_p\\
&= E_p^*H_G^*S_m^*(T_GT_G^*)^{-1}T_GE_p\\
&= E_p^*T_G^* S_mS_m^*(T_GT_G^*)^{-1}T_GE_p.
\end{align*}
Next we use $S_mS_m^*= I -E_mE_m^*$ to show that
\begin{align}
B_2&=E_p^*T_G^* (T_GT_G^*)^{-1}T_GE_p-E_p^*T_G^* E_mE_m^*(T_GT_G^*)^{-1}T_GE_p\nn\\
&=E_p^*T_G^* (T_GT_G^*)^{-1}T_GE_p-G_0^*E_m^*(T_GT_G^*)^{-1}T_GE_p. \label{altB2}
\end{align}
Recall (see the final part of Lemma \ref{L:MatRes}) that $\tht_0^*G_0^*=0$, and hence $\tht_0^*B_2=\tht_0^*E_p^*T_G^* (T_GT_G^*)^{-1}T_GE_p$. Since $A=B_1$ and $\tht_0\tht_0^*$ is given by \eqref{deftht0}, we conclude that
\begin{align*}
&\tht_0\tht_0^*\Big(I_p +E_p^*H_G^*(T_GT_G^*)^{-1}H_G E_p\Big)\tht_0\tht_0^*=\\
&\hspace{1cm}=\tht_0\tht_0^*\tht_0\tht^*+\tht_0\tht_0^*E_p^*H_G^*(T_GT_G^*)^{-1}H_G E_p\tht_0\tht_0^*\\
&\hspace{1cm}=\tht_0\tht_0^*\tht_0\tht^*+\tht_0\tht_0^*(A-B)\\
&\hspace{1cm}=\tht_0\tht_0^*\tht_0\tht^*+\tht_0\tht_0^*(A-B_1+B_2)\\
&\hspace{1cm}=\tht_0\tht_0^*\tht_0\tht^*+\tht_0\tht_0^*B_2 \\
&\hspace{1cm}=\tht_0\tht_0^*\tht_0\tht^*+\tht_0\tht_0^*E_p^*T_G^* (T_GT_G^*)^{-1}T_GE_p \\
&\hspace{1cm}=\tht_0\tht_0^*\Big(I_p -E_p^*T_G^* (T_GT_G^*)^{-1}T_GE_p+ E_p^*T_G^* (T_GT_G^*)^{-1}T_GE_p\Big) \\
&\hspace{1cm}= \tht_0\tht_0^*.
\end{align*}
Hence $\tht_0\tht_0^*\Big(I_p +E_p^*H_G^*(T_GT_G^*)^{-1}H_G E_p\Big)\tht_0\tht_0^*=\tht_0\tht_0^*$.  But then, using that $\tht_0^*$ is surjective and $\tht_0$ is injective, we obtain \eqref{eqtht0*}, and the lemma is proved.
\end{proof}

We proceed by showing   that (C1) and (C2) are satisfied. Let $\ga_0, \ga_1, \ga_2, \cdots$ be given by \eqref{defgas}. Using the first part of Lemma \ref{lemtht1} and formula \eqref{eqtht0*} we obtain  for each $u\in \BC^{p-m}$ that
 \begin{align*}
\|\ga_j u\|^2&
= \inn{\ga_j^* \ga_j u}{u}
=\inn{\tht_0^*\Big(\sum_{i=0}^\iy Y_i^*Y_i\Big)\tht_0 u}{u}\\
&=\inn{\tht_0^*\Big(I_p +E_p^*H_G^*(T_GT_G^*)^{-1}H_G E_p\Big)\tht_0 u}{u}\\
&=\inn{u}{u}=\|u\|^2 \quad (j=0,1,2, \cdots).
\end{align*}
Thus  (C1)  holds.

Next, in order to derive  (C2), we   use the second part of  Lemma \ref{lemtht1} and the fact that $\tht_0^*G_0^*=0$. For $j>k$ this yields
\begin{align*}
\ga_j^*\ga_k&=\tht_0^*\Big(\sum_{i=0}^\iy Y_i^* Y_{i+j-k}\Big)\tht_0\\
&=-\tht_0^*G_0^*E_m^*(S_m^*)^{j-k-1}(T_GT_G^*)^{-1}H_G E_p \tht_0=0.
\end{align*}
It follows that $\im  \ga_j \perp \im  \ga_k$  for $j>k$. Interchanging the role of $j$ and $k$ then yields (C2).

Finally, we prove $\kr T_G = \im T_{\tht}$. Recall that $G_0\tht_0=0$ by the final part of Lemma \ref{L:MatRes}. Hence using \eqref{GYG0} we have
\[
G(z)\tht(z)=G(z)Y(z)\tht_0=G_0\tht_0=0 \quad(z\in \BD).
\]
This implies $T_GT_\tht=0$, and thus $\im T_{\tht}\subset \kr T_G$. To prove the reverse inclusion,  take $f=\begin{bmatrix}f_0& f_1&f_2&   \cdots \end{bmatrix}{}^\perp$ in $\kr T_G$, and put $F(z)=E_p(I-zS_p^*)^{-1}f$. Since $G(z)F(z)=0$ on $\BD$, the second part of \eqref{invGH2} shows that
\begin{align*}
F(z)&=\begin{bmatrix}\Xi(z) &\tht(z)  \end{bmatrix}\begin{bmatrix}
G(z)\\[.2cm]  H(z) \end{bmatrix}F(z)\\
&=\begin{bmatrix}\Xi(z) &\tht(z)  \end{bmatrix}\begin{bmatrix}
0\\[.2cm]  H(z)  {F(z)} \end{bmatrix}=\tht(z) H(z)F(z) \quad (z\in \BD).
\end{align*}
It follows that  $f= {T_{\tht}}T_Hf$, and thus $f\in \im  {T_{\tht}}$ which proves that $\kr T_G\subset \im {T_{\tht}}$, and therefore $\kr T_G = \im T_{\tht}$. This completes the direct  {proof of Proposition \ref{P:Theta}.}

\section{Proof of Theorem 1.1}\label{secWiener}\setcounter{equation}{0}

In this section we prove Theorem \ref{thmmain1}.  For that purpose we first derive the following lemma.

\begin{lem}\label{lemW1} Let   $G\in \sW_+^{m\ts p}$, and assume that  $T_GT_G^*$ is strictly positive. Then $(T_GT_G^*)^{-1}$ maps $\ell_+^1(\BC^m)$ into itself.
\end{lem}

\begin{proof}[\bf Proof.]
We split   the proof   into five  parts. In the first part  we review  a few general facts about Toeplitz and Hankel operators (cf., Sections 2.1--2.3 in \cite{BS90} and  Chapter XXIII in \cite{GGK2}),   and we recall  an inversion formula from \cite{FKR2a}.

\smallskip\noindent\textsc{Part 1.}
Let  $F$ belong to  the Wiener space $ \sW_+^{r\ts s}$. Then the  Toeplitz operator $T_F$ and the Hankel operator $H_F$ both map  $\ell_+^1 (\BC^s)$ (seen as   a linear sub-manifold of $\ell_+^2(\BC^s)$) into $\ell_+^1(\BC^r)$ (seen as a linear sub-manifold of $\ell_+^2 (\BC^r)$). Moreover, the induced operators acting between these $\ell_+^1$ spaces are bounded too. Furthermore, $H_F$  is compact as  an operator  from $\ell_+^2 (\BC^s)$ into $\ell_+^2(\BC^r)$ as well as   when viewed as  an operator  from $\ell_+^1 (\BC^s)$ into $\ell_+^1(\BC^r)$ (cf.,  \cite[Sections 2.1]{BS90}). Finally, if $u$ is a ${s\ts t}$ matrix, then the  functions $\va$ and $\psi$ given by
\[
\va(z)=E_r^*(I-zS_r^*)^{-1} T_FE_s u \ands \psi(z)=E_r^*(I-zS_r^*)^{-1}H_F  E_s u
\]
belong to the Wiener space $\sW_+^{r \ts t}$.

Next, we recall some facts from \cite[Section 2]{FKR2a}. Define $R=G G^*$.  See the last paragraph of the introduction for the definition of $G^*$.  Note that $R\in \sW^{m\ts m}$. The fact that $T_GT_G^*$ is strictly positive implies that the matrix $R(z) $ is positive definite  for each $z\in \BT$, and hence the Toeplitz operator  $T_R$ acting on  $\ell_+^2 (\BC^m)$ is invertible. Moreover, see  \cite[Eq. (2.4)]{FKR2a}, we have
\begin{equation}\label{TRinv1}
(T_GT_G^*)^{-1}= T_R^{-1}+T_R^{-1}H_G(I-H_G^*T_R^{-1}H_G)^{-1}H_G^*T_R^{-1}.
\end{equation}

\smallskip\noindent\textsc{Part 2.}
Since $R$ belongs to $\sW^{m\ts m}$ and   $R(z) $ is positive definite  for each $z\in \BT$, the function  $R$ admits a a canonical spectral factorization (see Corollary 2.1 in \cite[Section III.2]{CG81}), that is,  $R=R_+^*R_+$ where $R_+$ belongs to $\sW_+^{m\ts m}$ and $\det R(z)\not =0$ for each $z$ in the closed unit disc. This implies that $T_R=T_{R_+^*} T_{R_+}$ and both $T_{R_+}$ and $T_{R_+^*}$ are  invertible. In fact,  $(T_{R_+})^{-1}= T_{R_+^{-1}}$  and $(T_{R_+^*})^{-1}=T_{(R_+^{-1})^*}$ are both Toeplitz operators.   We conclude that $T_R$ is invertible  and that its inverse is given by
\[
T_R^{-1}=(T_{R_+})^{-1}(T_{R_+^* })^{-1}.
\]
From the remarks in the first paragraph of the proof it then follows that the operators $(T_{R_+})^{-1}$ and $(T_{R_+^*} )^{-1}$ map $\ell_+^1 (\BC^m)$ into itself
and act as bounded linear operators on this space. Hence the same holds true for $T_R^{-1}$. Moreover $T_R^{-1}$, as an operator on $\ell_+^1 (\BC^m)$, is again invertible.

\smallskip\noindent\textsc{Part 3.}  From the final remark in the first paragraph of the first part  of the proof we know that the Hankel operator $H_G$ maps $\ell_+^1 (\BC^p)$ into $\ell_+^1 (\BC^m)$. An analogous  result holds true for $H_G^*$. To see this note that  $H_G^*=H_{G_*}$, where $G_*$ is the function in $\sW_+^{p\ts m}$ given by:
\[
G_*(z)=G(\bar{z})^*=G_0^*+zG_0^*+z^2G_2^*+\cdots  \quad (|z|\leq 1).
\]
Using the result of  Part 1 of the proof we conclude that $I-H_G^*T_R^{-1}H_G$ maps $\ell_+^1 (\BC^p)$ into itself and act as bounded linear operator  on this space.

\smallskip\noindent\textsc{Part 4.}  Put $M=I-H_G^*T_R^{-1}H_G$.  In this part we show that $M$ is invertible as an operator on  $\ell_+^1 (\BC^p)$. To do this we use the fact that $H_G$ acts as a compact operator from  $\ell_+^1 (\BC^p)$ to   $\ell_+^1 (\BC^m)$.  {It follows} that $M$  as an operator on $\ell_+^1 (\BC^p)$ is of the form identity operator plus a compact one. Hence $M$  as an operator on $\ell_+^1 (\BC^p)$ is a Fredholm operator of index zero.  {In order to} show that $M$  as an operator on $\ell_+^1 (\BC^p)$ is  {invertible, it then suffices} to prove that $M$   on $\ell_+^1 (\BC^p)$  is one-to-one. Take $h\in \ell_+^1 (\BC^p)$, and assume that $Mh=0$. Since $\ell_+^1 (\BC^p)$ is contained in $\ell_+^2 (\BC^p)$, it follows that $h\in \ell_+^2 (\BC^p)$. But on $\ell_+^2 (\BC^p)$ the operator $M$ is invertible. Thus $h=0$, and  $M$ is one-to-one on $\ell_+^1 (\BC^p)$. Therefore $I-H_G^*T_R^{-1}H_G$  is invertible as an operator on $\ell_+^1 (\BC^p)$.

\smallskip\noindent\textsc{Part 5.} The results of the preceding parts of the proof show that the operators appearing in \eqref{TRinv1} all map $\ell^1$  spaces into $\ell^1$  spaces, and hence  $(T_GT_G^*)^{-1}$ maps $\ell_+^1(\BC^m)$ into itself.
\end{proof}

\begin{proof}[\bf Proof of Theorem \ref{thmmain1}.]
We split the proof into three parts.

\smallskip
\noindent\textsc{Part 1.}  In this part we show that   the function  $Y$ defined by \eqref{defY}  has the desired properties.  First we show that $Y$ belongs to the Wiener space $ \sW_+^{p\ts p}$. To do this note that   $T_G^*=T_{G^*}$, and hence  $T_G^*$ maps  $\ell_+^1(\BC^m)$ into $\ell_+^1(\BC^p)$. But then  Lemma~\ref{lemW1} tells us that  $T_G^*(T_GT_G^*)^{-1}$ maps  $\ell_+^1(\BC^m)$ into $\ell_+^1(\BC^p)$.  Since $G\in \sW_+^{m\ts p}$, its Taylor coefficients $G_0, G_1, G_2, \ldots$ at zero are absolutely summable in norm,  and thus   we can use \eqref{defYj}  to  show that the same holds true for  the Taylor coefficients at zero of $Y$. Therefore $Y\in \sW_+^{p\ts p}$.

 Next we show that $\det Y(z)\not =0$ when $|z|\leq 1$.  For $|z|<1$  this follows from  Proposition \ref{propY2}. We shall prove that $\det Y(z)\not= 0$ for all $z\in \BT$ by contradiction. Assume  that there exists $\l\in \BT$ such that $\det Y(\l) =0$. Then  there exists $u\not = 0$ such that $Y(\l)u = 0$.  Since $G$ and $Y$ are Wiener functions,  $G$ and $Y$ extend continuously to $\BT$.  Thus the equality in \eqref{GYG0} from Proposition \ref{propY2} also holds for each $|z|=1$. It follows that  $G_0 u = G(\l)Y(\l)u =0$.  So $u\in \kr G_0$. From the final part of Lemma \ref{L:MatRes} we know that $\kr G_0 = \im \tht_0$.  But then $u = \tht_0 v$ for some $v\in \BC^{p-m}$, and $\tht(\l)v = Y(\l)\tht_0 v = Y(\l)u = 0$.  We obtain that   $v= I_{p-m}v= \Theta^*(\l)\Theta(\l)v =0$. This implies that $u =\Theta_0v =0$, which contradicts  our assumption that  $u\not = 0$. We conclude that  $\det Y(z)\not =0$ for all $|z|\leq 1$.

 By Wiener's theorem, the fact that $ Y\in \sW_+^{p\ts p}$ and $\det Y(z)\not =0$ for all $|z|\leq 1$ implies that $Y^{-1}$ also belongs to  $\sW_+^{p\ts p}$. Finally, formula \eqref{invY}   follows from  \eqref{invY2}. Thus $Y$ has all properties mentioned in the first paragraph of Theorem \ref{thmmain1}.

\smallskip
\noindent\textsc{Part 2.}  In this part we deal with items (i)--(iii).
Note that Proposition \ref{propY2} and the final part of Lemma \ref{L:MatRes} show  that the statements in items (i) and (ii) in Theorem \ref{thmmain1} hold true,  {noting that $Y,\, Y^{-1}\in\sW_+^{p\ts p}$ implies that the functions $\Xi$, $\tht$ and $H$ are analytic Wiener functions as well}. Furthermore, item (iii) follows from Proposition \ref{P:YH} and the fact that $G$, $H$, and $Y$ extend to continuous functions on $\BT$. This proves items (i)--(iii) Theorem \ref{thmmain1}.

\smallskip
\noindent\textsc{Part 3.} It remains to prove  the statements in the final paragraph of
Theorem \ref{thmmain1}. Put $\Xi(z)=Y(z)\Xi_0$. Clearly  $\Xi \in \sW_+^{p\ts m}$.  {Using \eqref{GYG0} from Proposition \ref{propY2}}, we have
\[
G(z)\Xi(z)=G(z)Y(z)\Xi_0=G_0\Xi_0\quad  (z\in \BD).
\]
Among other things, equality \eqref{invXiTheta} shows that $G_0\Xi_0=I_m$. It follows that $\Xi$ is a solution to the  {Wiener-Bezout} problem \eqref{corona}. From  the equality \eqref{invXiTheta} it also follows that $G_0\tht_0=0$. Hence for $X$ given by \eqref{allsolW} with $V$ belonging to $\sW_+^{(p-m)\ts m}$ we have
\begin{align*}
G(z)X(z)&= G(z)Y(z)(\Xi_0+\tht_0 V(z))\\
&=G_0\Xi_0+G_0\tht_0V(z)=I_m \quad   (z\in \BD).
\end{align*}
Note that $X$ given by \eqref{allsolW} belongs to $\sW_+^{p\ts m}$, and thus all $X$ given by \eqref{allsolW} are solutions to the  {Wiener-Bezout} problem  associated with $G$.

We proceed by  proving \eqref{H2idW}. To do this let $V\in \sW_+^{(p-m)\ts m}$, and let $X$ be given by \eqref{allsolW}. From Lemma \ref{lemXi2} we know that $\Xi$  is  given by \eqref{altdefXi}. This implies that
\[
\im T_\Xi E_m=\im T_G^*( T_G T_G^*)^{-1}E_m\subset \im T_G^*=(\im T_\tht)^\perp.
\]
Thus for each $u\in \BC^m$ the vector   $T_\Xi E_mu$ is orthogonal to $\im T_\tht$
Using this orthogonality  we have
\begin{align}
\|X(\cdot)u\|^2_{H^2_p}&=\|T_X E_m u\|^2_{\ell_+^2(\BC^p)}=\|T_\Xi E_mu + T_\tht T_V E_m u\|^2_{\ell_+^2(\BC^p)}\nn \\
&=\|T_\Xi E_m u\|^2  +\|T_\tht T_V E_m u\|^2_{\ell_+^2(\BC^p)}\nn \\
&=\|T_\Xi E_m u|^2+\| T_V E_m  u\|^2_{\ell_+^2(\BC^p)}\nn \\
&= \|\Xi(\cdot)u\|^2_{H^2_p}+ \|V(\cdot)u\|^2_{H^2_p},
\label{prH2idW}
\end{align}
which proves \eqref{H2idW}.

Finally, let $X\in \sW_+^{p\ts m}$ be a solution  to the  {Wiener-Bezout} problem  associated with $G$.  Thus $G(z)X(z)=I_m$ for $z\in \BD$. Define $V(z)=H(z)X(z)$, $z\in \BD$, where $H$ is defined in
item (iii) of Theorem \ref{thmmain1}.  Then $V$ belongs to the Wiener space $W_+^{(p-m)\ts m}$,  and formula \eqref{invGH1} shows that
\begin{align*}
X(z)&=\begin{bmatrix} \Xi(z)& \tht(z)\end{bmatrix}\begin{bmatrix} G(z)\\ H(z)\end{bmatrix}X(z)\\
&=  \Xi(z)G(z)X(z)+\tht(z)H(z)X(z)=\Xi(z)+\tht(z)V(z) \quad (|z|\leq 1).
\end{align*}
Using the formulas for  $ \Xi(z)$ and  $\tht(z)$ in \eqref{defXiTheta} we see that $X$ admits the representation \eqref{allsolW}.
\end{proof}

\begin{remark}\label{remHwtH}
In the Wiener setting the function  $H$ defined in item (iii) of  Theorem \ref{thmmain1} and the function $H$ defined by \eqref{TolH} are equal. To be more precise, put
\begin{align*}
H(z)&=(\tht_0^*\tht_0)^{-1}\tht_0^*(I_p-\Xi_0G_0)  Y(z) ^{-1} \quad  (|z|\leq 1), \\
\wt{H}(\z)&=\tht^*\big(\z)(I_p-\Xi(z)G(\z)\big) \quad (|\z|=1).
\end{align*}
Then $H=\wt{H}$. To see this fix $|\z|=1$.  {According to \eqref{invXiTheta} we have
\[
H(\z)Y(\z)\mat{cc}{\tht_0&\Xi_0}=H_0 \mat{cc}{\tht_0&\Xi_0}
=\mat{cc}{I_{p-m}& 0}.
\]}
On the other hand, according  item (i) in Theorem \ref{thmmain1} we have $G(\z)Y(\z)=G_0$. Furthermore, by definition, $\Xi(\z)=Y(\z)\Xi_0$. It follows that
\begin{align*}
\wt{H}(\z)Y(\z)&=\tht^*\big(\z)(Y(\z)-\Xi(\z)G(\z)Y(z)\big) \\
&=\tht^*\big(\z)(Y(\z)-\Xi(z)G_0\big)=\tht^*(\z)Y(\z)\big(I-\Xi_0G_0\big).
\end{align*}

Again using \eqref{invXiTheta}, we obtain  $G_0\tht_0=0$ and $G_0\Xi_0=I_m$, such that
\[
(I-\Xi_0G_0)\mat{cc}{\tht_0&\Xi_0}=\mat{cc}{\tht_0&0}.
\]
This yields
\begin{align*}
\wt{H}(\z)Y(\z)\mat{cc}{\tht_0&\Xi_0}&
=\tht^*(\z)Y(\z)\mat{cc}{\tht_0&0}
=\tht^*(\z)\mat{cc}{\tht(\z)&0}\\
&=\mat{cc}{\tht(\z)&0}=\mat{cc}{I_{p-m}&0}.
\end{align*}
Since $\mat{cc}{\tht_0&\Xi_0}$ and $Y(\z)$ both are invertible, we obtain that $H(\z)=\wt{H}(\z)$.  But $\z$ is an arbitrary point on $\BT$. Therefore, $H=\wt{H}$.
\end{remark}

\section{Solutions to the  {$H^2$-Bezout} problem}
\label{HiyH2}\setcounter{equation}{0}

Let $G\in H^\iy_{m\ts p}$ and assume $T_GT_G^*$ is strictly positive. If $G$ is not in $\sW_+^{m\ts p}$, then the function $\Xi$ defined in \eqref{defXiTheta} will, in general, not be in $\sW_+^{p\ts m}$, and hence not a solution to the  {Wiener-Bezout} problem associated with $G$. However, by Propositions \ref{propY2} and \ref{P:YH},  {the function} $\Xi$ is still analytic on $\BD$ and satisfies $G(z)\Xi(z)=I_m$ for each $z\in\BD$. It turns out that r$\Xi$ is in $H^2_{p\ts m}$ and hence a solution to the  {$H^2$-Bezout} problem associated with $G$. In fact, extending the description of all solutions to the  {Wiener-Bezout} problem of Theorem \ref{thmmain1} via \eqref{allsolW} to one where $V$ is taken from $H^2_{(p-m)\ts m}$, all solutions to the  {$H^2$-Bezout} problem are obtained, even if $G\not\in\sW_+^{m\ts p}$. The details are given in the following theorem.

\begin{thm}\label{T:H2cor}
Let $G\in H^\iy_{m\ts p}$ such that $T_GT_G^*$ is strictly positive. Define the functions $\Xi$ and $\tht$ by \eqref{defXiTheta}, with $\Xi_0$ and $\tht_0$ as in \textup{(M1)} and \textup{(M2)}. Then $\Xi\in H^2_{p\ts m}$, $\tht\in H^\iy_{p\ts (m-p)}$ is inner with $\im T_\tht=\kr T_G$, and for any $V\in H^2_{(p-m)\ts m}$ the function
\begin{equation}\label{H2sol}
X(z)=\Xi(z)+\tht(z)V(z)\quad (z\in\BD)
\end{equation}
is a solution to the  {$H^2$-Bezout} problem associated with $G$. Moreover, all solutions are obtained in this way. Furthermore, for $X$ given by \eqref{H2sol}, with $V$ in $H^2_{(p-m)\ts m}$, we have
\begin{equation}\label{H2LeastSquare}
\|X(\cdot)u\|^2_{H^2_p}=\|\Xi(\cdot)u\|^2_{H^2_p}+\|V(\cdot)u\|^2_{H^2_{p-m}}
\quad (u\in\BC^m).
\end{equation}
In particular, $\Xi$ is the last square solution to the  {$H^2$-Bezout} problem associated with $G$.
\end{thm}

Theorem \ref{T:H2cor} gives a variation on the last part of our main result, Theorem \ref{thmmain1} above, under the weaker assumption $G\in H^\iy_{m\ts p}$. Variations on the other claims made in Theorem \ref{thmmain1}, e.g., \eqref{invY} and items (i)--(iii), under this weaker assumption were proved in Propositions \ref{propY2}, \ref{P:YH} and \ref{P:Theta} above.

We shall first prove the next proposition, which contains the key observation needed in the proof of Theorem \ref{T:H2cor}.

\begin{prop}\label{P:YY-1H2}
Let $G\in H^\iy_{m\ts p}$ such that $T_GT_G^*$ is strictly positive. Then the function $Y$ defined by \eqref{defY} as well as the function $Y(\cdot )^{-1}$ are in $H^2_{p\ts p}$. In particular, $\det Y(z)\not =0$ for almost every $z\in\BT$.
\end{prop}

In order to prove Proposition \ref{P:YY-1H2} we require some additional notation. Let $F\in H_{r\ts s}^2$. Then $F$ admits a Taylor expansion
\begin{equation}\label{F-Taylor}
F(z)=F_0+zF_1+z^2F_2+\cdots  \quad (z\in\BD)
\end{equation}
and induces a bounded operator
\begin{equation*}
\ga_F=\mat{c}{F_0\\F_1\\\vdots}:\BC^s\to\ell^2_+(\BC^r)\ \ \mbox{with}\ \
\|\ga_Fu\|_{\ell^2_+(\BC^r)}=\|F(\cdot)u\|_{H^2_r}\ \ (u\in\BC^s).
\end{equation*}
In fact, an analytic $r\ts s$ matrix function $F$ as in \eqref{F-Taylor} is in $H_{r\ts s}^2$ if and only if $\ga_F$ above induces a bounded operator from $\BC^s$ into $\ell^2_+(\BC^r)$. On the other hand, if $K$ is a bounded operator from $\BC^s$ into $\ell^2_+(\BC^r)$, then $K=\ga_F$ for some $F\in H_{r\ts s}^2$; in this case $F_n:=E_r^* S_r^{*n}K$ is the $n$-th Taylor coefficient of $F$.

With $F\in H_{r\ts s}^2$ we associate a function ${F_*}$ defined by
\begin{equation}
\label{defF*}
F_*(z)=F(\bar{z})^*=F_0^*+zF_1^*+z^2F_2^*+\cdots  \quad (|z|<1).
\end{equation}
Then $F_*\in H^2_{s\ts r}$. However, $\ga_F$ and $\ga_{F_*}$ need not have the same operator norm. In contrast, for $F$ in $H_{r\ts s}^\iy$, we have ${F_*}\in H_{s\ts r}^\iy$ and $\|T_F\|=\|F\|_\iy=\|{F_*}\|_\iy=\|T_{\wtil{F}}\|$.

\begin{proof}[\bf Proof  of Proposition \ref{P:YY-1H2}]
Identifying $\ell^2_+(\BC^p)$ with $\BC^p\oplus \ell^2_+(\BC^p)$, we see that
\[
\ga_Y=\mat{c}{I_p\\
-T_G^*(T_GT_G^*)^{-1}H_GE_p}:\BC^p\to \mat{c}{\BC^p\\ \ell^2_+(\BC^p)},
\]
which is clearly bounded as an operator from $\BC^p$ into $\ell^2_+(\BC^p)$. Hence $Y\in H^2_{p\ts p}$.

Now define $F$ on $\BD$ by
\[
F(z)=I+zE_p^*(I-zS_p^*)^{-1}H_G^*(T_GT_G^*)^{-1}T_GE_p\quad (z\in\BD).
\]
Again identifying $\ell^2_+(\BC^p)$ with $\BC^p\oplus \ell^2_+(\BC^p)$, this in turn shows that
\[
\ga_{F}=\mat{c}{I_p\\ H_G^*(T_GT_G^*)^{-1}T_GE_p} :\BC^p\to \mat{c}{\BC^p\\ \ell^2_+(\BC^p)}.
\]
Hence $\ga_{F}$ is bounded, and thus $F$ is in $H^2_{p\ts p}$. Moreover, by \eqref{invY2}, we have
\[
F_*(z)=I+zE_p^*T_G^*(T_GT_G^*)^{-1}H_G(I-zS_p)^{-1}E_p=Y(z)^{-1}\quad (z\in\BD).
\]
Since $F\in H^2_{p\ts p}$, we obtain that $F_*\in H^2_{p\ts p}$.
Hence $Y(\cdot )^{-1}$ is in $H^2_{p\ts p}$.

Since both $Y$ and $F$ are in $H^2_{p\ts p}$ with $Y(z)F(z)=I_p$ for every $z\in\BD$, the non-tangential limits of $F$ and $Y$ exist a.e.\ on $\BT$ and, by continuity, the identity $Y(z)F(z)=I_p$ extends to all points on $\BT$ where the non-tangential limits of both exist. Thus $Y(z)$ is invertible for almost every $z\in\BT$.
\end{proof}

\begin{proof}[\bf Proof of Theorem \ref{T:H2cor}.]
Since $ Y(\cdot)^{-1}$ is in $ H^2_{p\ts p}$ and $\Xi(z)=Y(z)\Xi_0$, $z\in\BD$, we have $\Xi\in H^2_{p\ts m}$.  The claim regarding $\tht$ follows from Proposition \ref{P:Theta}. Let $V\in  H^2_{(p-m)\ts m}$. The preceding observations about $\Xi$ and $\tht$ show that the function $X$ given by \eqref{H2sol} is in $H^2_{p\ts m}$. By \eqref{invGH2} we have $G(z)\Xi(z)=I_m$ and $G(z)\tht(z)=0$, so that $G(z)X(z)=I_m$ for each $z\in\BD$. Hence for each $V\in  H^2_{(p-m)\ts m}$, the function $X$ given by \eqref{H2sol} is a solution to the  {$H^2$-Bezout} problem associated with $G$.

Next we  show that all solutions to the  {$H^2$-Bezout} problem associated with $G$ are obtained through \eqref{H2sol}.   To do this, we first note that   \eqref{altdefXi} implies that   $\ga_\Xi=T_G^*(T_GT_G^*)^{-1}E_m$.

Now let $X\in H^2_{p\ts m}$ be a solution to \eqref{corona}. Then $\ga_{X}$ is bounded and \eqref{corona} translates to $T_G \ga_{X}=E_m$.  We thus obtain  that
\[
T_G^*(T_GT_G^*)^{-1}T_G \ga_{X}=T_G^*(T_GT_G^*)^{-1}E_m=\ga_\Xi.
\]
Note that $T_G^*(T_GT_G^*)^{-1}T_G=I-P_{\kr T_G}=I-T_\tht T_\tht^*$. Hence
\[
\ga_{X}=T_G^*(T_GT_G^*)^{-1}T_G\ga_{X} + T_\tht T_\tht^* \ga_{X}
=\ga_\Xi +T_\tht \ga_V,
\]
where $V\in H^2_{(p-m)\ts m}$ is determined by $\ga_V=T_\tht^* \ga_{X}$. The above identity implies $X$ is given by \eqref{H2sol} with $V\in H^2_{(p-m)\ts m}$ such that $\ga_V=T_\tht^* \ga_{X}$.

It remains to derive the identity  \eqref{H2LeastSquare}.  But this can be done  by using the same argumentation as in the proof of Theorem \ref{thmmain1} (see \eqref{prH2idW}); we omit the details.
\end{proof}

Let $G\in H^\iy_{m\ts p}$ be such that $T_GT_G^*$ is strictly positive. Then $Y\in H^2_{p\ts p}$, by Proposition \ref{P:YY-1H2}. Now assume $Y\in H^\iy_{p\ts p}$. In that case all solutions to the $H^\iy$-corona problem  associated with $G$ are given by formula  \eqref{allsolW}. More precisely we have the following proposition.

\begin{prop}\label{propHinf}
Let $G\in H^\iy_{m\ts p}$ be such that $T_GT_G^*$ is strictly positive, and assume that the function  $Y$  defined by \eqref{defY} belongs to $Y\in H^\iy_{p\ts p}$. Then $Y$ is invertible outer, and all  solutions to the $H^\iy$-corona problem are given by \eqref{allsolW} where the free parameter is any $V\in H^\iy_{(p-m)\ts m}$.
\end{prop}

\begin{proof}[\bf Proof]
Assume the function $Y$ defined by \eqref{defY} belongs to $Y\in H^\iy_{p\ts p}$. Then $\Xi$ defined in \eqref{defXiTheta} is in $H^\infty_{p\times p-m}$, and hence $\Xi$ is a solution to the $H^\iy$-corona problem associated with $G$. Since the inner function $\tht$ is in $H^\infty_{p\times p-m}$, we obtain that $X$ defined by \eqref{allsolW} is in $H^\iy_{p\ts m}$ whenever the parameter $V$ is in $H^\iy_{(p-m)\ts m}$. Hence the map $V\mapsto X$ in \eqref{allsolW} produces solutions to the $H^\iy$-corona problem when restricted to parameters $V\in H^\iy_{(p-m)\ts m}$.

Next we show that all solutions to the $H^\iy$-corona problem associated with $G$ are obtained by \eqref{H2sol} when the parameters $V$ are restricted to $H^\iy_{(p-m)\ts m}$. Assume $X\in H^\iy_{p\ts m}$ satisfies \eqref{corona}. Then $X$ is also a $H^2$-solution of \eqref{corona}, and hence $X$ is given by \eqref{H2sol} for some $V\in H^2_{(p-m)\ts m}$. It remains to show that $V\in H^\iy_{(p-m)\ts m}$. The latter follows by considering the values of $V$ in $\BT$, and noting that for almost every $\z\in\BT$ we have
\begin{align*}
\|V(\z)\|=\|\tht(\z)V(\z)\|&=\|X(\z)-\Xi(\z)\|\leq \\
& \leq\|X(\z)\|+\|\Xi(\z)\|\leq \|X\|_\iy+\|\Xi\|_\iy.
\end{align*}
Hence $\|V\|_\iy\leq \|X\|_\iy+\|\Xi\|_\iy<\infty$, and thus $V\in H^\iy_{(p-m)\ts m}$.

We conclude with the proof that $Y\in H^\iy_{p\ts p}$ implies that the function $Y(\cdot)^{-1}$ also belongs to  $H^\iy_{p\ts p}$, i.e., that $Y$ is invertible outer. To see that this is the case, recall from Section \ref{secWiener}, that the function $H$ defined by \eqref{defH}, on the circle is given by
\[
H(z)=\tht(z)^*(I-\Xi(z)G(z))\quad (\mbox{a.e. }z\in\BT).
\]
(Note that this observation does not require $G\in\sW_{m\ts p}^+$.) Since $\tht$, $\Xi$ and $G$ are all $H^\iy$-functions, it follows that $H$ is essentially bounded on $\BT$, and thus $H\in H^\iy_{(p-m) \ts p}$. This implies that the function
\[
z\mapsto \mat{c}{G(z)\\ H(z)}=\mat{c}{G_0\\ H_0}Y(z)^{-1}
\]
is in $H^{\iy}_{p\ts p}$. By the invertibility of $\sbm{G_0\\ H_0}$, it follows that $Y^{-1}$ is in $H^\iy_{p\ts p}$, as claimed.
\end{proof}

Note that the description of the solutions to the $H^\iy$-corona problem associated with $G$ obtained in this way is much simpler than the one obtained in \cite[Remark 4.1]{FtHK-8}. However, while the description \eqref{H2sol} has a favorable behavior with respect to the $H^2$-norm (see \eqref{H2LeastSquare}), there is no clear connection between the supremum norms of the parameter $V$ and the solution $X$ related through \eqref{H2sol}, making it a less suitable way to describe solutions with an additional bound on the supremum norm.

Furthermore, at this stage it is unknown whether the function $Y$ in \eqref{defY} belongs to $H^\iy_{p\ts p}$, or not. Can it happen that $Y$ does not belong to $H^\iy_{p\ts p}$, and if so, under what conditions on $G$ does $Y$ belong to $H^\iy_{p\ts p}$?  Theorem \ref{thmmain1} yields that $Y\in H^\iy_{p\ts p}$ whenever $G\in \sW_+^{m\ts p}$.

\section{Concluding remarks}
\setcounter{equation}{0}

\begin{remark}\label{R:Gconst}
Let us assume  that the $m\ts p$ matrix function $G$  is a constant function, that is, $G(z)=G_0$ for all $z\in \BD$. In that case the Bezout-corona equation \eqref{corona} reduces to
\begin{equation}\label{constG1}
G_0X(z)=I_m.
\end{equation}
This equation has a solution if and only if  the $m\ts p$ matrix $G_0$ is right invertible, and in that case   a straightforward application of Theorem \ref{thmmain1} shows that all Wiener class solutions of equation  \eqref{constG1}
are given by
\begin{equation}\label{solconstG1}
X(z)=G_0^*(G_0 G_0^*)^{-1}+\t_0 V(z), \qquad  V\in \sW_+^{(p-m)\ts m}.
\end{equation}
Here $\t_0$ is an isometry mapping $\BC^{p-m}$ onto $\kr G_0$. This result can also be derived directly by elementary linear algebra, using that  $G_0^*(G_0 G_0^*)^{-1}$ is the Moore-Penrose right inverse of $G_0$. Note the definition of $\t_0$ implies that $\t_0 \t_0^*$ is the orthogonal  projection of  $\BC^p$ onto $\kr G_0$, and hence (cf., \eqref{deftht0})  we have
\[
\t_0 \t_0^*=I_p-G_0^*(G_0 G_0^*)^{-1}G_0 \ands \kr \t_0 =\{0\}.
\]
Finally, in this particular case the function defined by \eqref{defY} is just identically equal to the $p\ts p$ identity matrix.
\end{remark}

\begin{remark}\label{R:allsol2} Let $G\in \sW_+^{m\ts p}$, and assume that  $T_GT_G^*$ is strictly positive. Let $Y\in  \sW_+^{p\ts p}$ be given by \eqref{defY}. Then $\det Y(z)\not =0$ and $G(z)=G_0 Y(z)^{-1}$ for each $|z|\leq 1$. Hence equation \eqref{corona} can be rewritten as
\[
G_0\Big(Y(z)^{-1}X(z)\Big)=I_m.
\]
But then we can apply the result of the previous remark to show that the set of all Wiener solutions to  the Wiener-Bezout problem defined by $G$ is given by
\[
X(z)=Y(z)\Big(G_0^*(G_0 G_0^*)^{-1}+\t_0 V(z)\Big),  \qquad  V\in \sW_+^{(p-m)\ts m}.
\]
The above representation of the set of all Wiener solutions differs from and is less informative than the  one given by \eqref{allsolW}.  For instance, in general, the function $Y(\cdot) G_0^*(G_0 G_0^*)^{-1}$ is not the least {squares solution.} 
\end{remark}

\begin{remark}\label{R:Yinv2a} If $G$ is a polynomial, then the function $Y(\cdot)^{-1}$ is also a polynomial and its degree is less than or equal to the degree of $G$. This fact  is a corollary of  formula \eqref{invY} in Theorem \ref{thmmain1}. However, in general,  the assumption $G$ is a polynomial does not imply that $Y$ is a polynomial.  To see this we take $G(z)=\begin{bmatrix} 1+z & -z \end{bmatrix}$ as in Example 1 in \cite[Section 5]{FKR2a} and show that for this specific choice of $G$ the function $Y$ is given by
\begin{align}\label{Yex}
Y(z)={\frac{1}{1+zq}\mat{cc}{1-(1-q)z &  z  \\ -(1-q)z  & 1+z}}.
\end{align}
Here $q=\half(3-\sqrt{5})\in(0,1/2)$, which satisfies
\begin{align}\label{qids}
q^2-3q+1=0,\quad \sqrt{q}=1-q \ands \frac{q(1-q)}{1-2q}=1.
\end{align}
We compute \eqref{Yex} via the formula for the Taylor coefficients of $Y$ given in \eqref{defYj}. For this purpose we rewrite the right hand side in the second identity of \eqref{defYj} as $-T_G^*(T_GT_G^*)^{-1}H_G E_2$, and we compute this operator following the approach of \cite{FKR2a}.  Recall (see \cite[Eq. (2.4)]{FKR2a} or \eqref{TRinv1}) that
\[
(T_GT_G^*)^{-1}= T_R^{-1}+T_R^{-1}H_G(I-H_G^*T_R^{-1}H_G)^{-1}H_G^*T_R^{-1},
\]
where $R=GG^*$. In the present example, where $G(z)=\begin{bmatrix} 1+z & -z \end{bmatrix}$, we have $R(z)=G(z)G(1/\bar{z})^*=3+z+z^{-1}$, and
hence $R(z)$ is strictly positive on $\BT$. If follows that  $R$ admits an outer spectral factorization, namely
\[
R(z)=\phi(1/\bar{z})^*\phi(z), \ \mbox{with $\phi(z)=q^{-\half}(1+zq)$ and $q=\half(3-\sqrt{5})$}.
\]
Furthermore,  $R(z)^{-1}=\psi(z)\psi(1/\bar{z})^*$ with
\[
\psi(z)=\sqrt{q}(1+zq)^{-1}=\sum_{\nu=0}^\infty \sqrt{q}(-q)^\nu z^\nu, \quad (z\in\BD).
\]
It follows that $(T_R)^{-1}=T_\psi T_\psi^*$, and
\begin{align*}
T_\psi=\sqrt{q}\mat{cccc}{v &Sv&S^2v&\cdots}\ \mbox{with}\  v=\mat{ccccc}{1&-q&(-q)^2&(-q)^3&\cdots}^*.
\end{align*}
Here $S$ denotes the forward shift operator on $\ell^2_+:=\ell^2_+(\BC)$. Since $G(z)=\begin{bmatrix} 1+z & -z \end{bmatrix}$, we have
\begin{align*}
T_G=\mat{ccccccc}{1&0&0&0&0&0&\cdots\\ 1&-1&1&0&0&0&\cdots\\ 0&0&1&-1&1&0&\cdots\\ \vdots&\vdots&\vdots&\vdots&\vdots&\vdots&},\quad
H_G=\mat{ccccc}{1&-1&0&0&\cdots\\ 0&0&0&0&\cdots\\ 0&0&0&0&\cdots\\ \vdots&\vdots&\vdots&\vdots&}.
\end{align*}
We then obtain
\[
T_R^{-1}H_G=T_\psi T_\psi^* H_G=\sqrt{q} T_\psi H_G= q\mat{cccccc}{v&-v&0&0&0&\cdots}.
\]
Identifying $\ell^2_+$ with $\BC^2\oplus\ell^2_+$, we get
\small{
\[
(I-H_G^* T_R^{-1} H_G)^{-1}
=\mat{ccc}{1-q & q &0\\ q & 1-q &0 \\  0& 0& I_{\ell^2_+}}^{-1}
=\mat{ccc}{1/q & -1/\sqrt{q} &0\\ -1/\sqrt{q} & 1/q &0 \\  0& 0& I_{\ell^2_+}}.
\]
}
Putting the above computations together yields
\begin{align*}
(T_GT_G^*)^{-1}
&=T_R^{-1}+q^2\mat{cc}{v & -v}\mat{cc}{1/q & -1/\sqrt{q}\\ -1/\sqrt{q} & 1/q} \mat{c}{v^*\\ -v^*}\\
&=T_{R}^{-1}+v \mat{cc}{1 &-1}\mat{cc}{q & -q\sqrt{q}\\ -q\sqrt{q} & q}\mat{c}{1\\-1}v^*\\
&=T_{R}^{-1}+v \mat{cc}{1 &-1}\mat{cc}{q & -(1-2q)\\ -(1-2q) & q}\mat{c}{1\\-1}v^*\\
&=T_{R}^{-1}+2(1-q)vv^* =T_{R}^{-1}+2\sqrt{q}\,vv^*.
\end{align*}
In the third step we used $q\sqrt{q}=q(1-q)=1-2q$ which follows from  the second and third identity in \eqref{qids}. Combining the formula for  $(T_GT_G^*)^{-1}$ with the one for $T_R^{-1}H_G$ yields
\begin{align*}
(T_GT_G^*)^{-1}H_G E_2
&=T_{R}^{-1}H_G E_2+2\sqrt{q}\,vv^*H_G E_2\\
&=q\mat{cc}{v&-v}+2\sqrt{q}\mat{cc}{v& -v}\\
&=q({1+2/\sqrt{q}})\mat{cc}{v& -v}=q(1-2q)^{-1}\mat{cc}{v& -v}.
\end{align*}
To see that the latter identity holds, note that $1-3q+q^2=0$ implies $3-q=1/q$, so that together with $q\sqrt{q}=1-2q$ we obtain
\[
1+2/\sqrt{q}=\frac{\sqrt{q}+2}{\sqrt{q}}=\frac{(1-q)+2}{\sqrt{q}}=\frac{3-q}{\sqrt{q}}=\frac{1}{q\sqrt{q}}=(1-2q)^{-1}.
\]
Next note that  $T_G^*v=vG(-q)^*.$ By \eqref{defYj} we then obtain that for $\nu=1,2,\ldots$ that
\begin{equation*}
Y_\nu=\frac{-q}{1-2q}(-q)^{\nu-1} N=\frac{(-q)^\nu}{1-2q}N,
\end{equation*}
where
\[
N=\mat{cc}{G(-q)^* & -G(-q)^*}=\mat{cc}{\sqrt{q}&-\sqrt{q}\\ q & -q}
=\mat{cc}{1-q&-(1-q)\\ q & -q}.
\]
Hence
\begin{align*}
&Y(z)
=I_2+\sum_{\nu=1}^\infty \frac{(-q)^\nu}{1-2q}N z^\nu=I_2+\frac{-qz}{(1-2q)(1+zq)}N\\
&\ \ =\frac{1}{(1-2q)(1+zq)}\mat{cc}{(1-2q)(1+zq)- q(1-q)z &  q(1-q)z  \\ -q^2z  & (1-2q)(1+zq)+zq^2}\\
&\ \ =\frac{1}{(1-2q)(1+zq)}\mat{cc}{(1-2q)-q^2z &  (1-2q)z  \\ -q^2z  & (1-2q)(1+z)}\\
&\ \ = {\frac{1}{1+zq}\mat{cc}{1-(1-q)z &  z  \\ -(1-q)z  & 1+z}}.
\end{align*}
{Here we used $(1-2q)(1+zq)=1-2q+qz-2q^2z$ and $q(1-q)=1-2q$ in the last but one identity, and $q^2/(1-2q)=q(1-q)^2/(1-2q)=1-q$ (which follows from the last two identities in \eqref{qids}) in the last identity.} Hence we obtain $Y$ is given by \eqref{Yex} as claimed.

Using similar calculations as in the previous remark one can prove   that for  $G(z)=\begin{bmatrix} 1+z & -z \end{bmatrix}$  the matrices $\Xi_0$ and $\tht_0$  are given by
\[
\Xi_0= \frac{q}{1-2q}\begin{bmatrix} 1-q \\ q \end{bmatrix} \ands \tht_0 = \frac{1}{\sqrt{q}}\begin{bmatrix}0\\ 1\end{bmatrix},
\]
where $q=\half(3-\sqrt{5})$. Note, however, that these formulas for  $\Xi_0$ and $\tht_0$ can be derived from  earlier results   in \cite{FKR2a} and \cite{FKR2b}. Indeed, $\Xi_0$ is obtained by taking $z=0$ in $X(z)$ on \cite[Page 414]{FKR2a}, and the formula for $\tht_0$ is obtained by taking $z=0$ in the formula for $\tht(z)$ appearing in the final paragraph of~\cite{FKR2b}.
\end{remark}

\end{document}